\documentclass{amsart}

\usepackage{amsfonts}
\usepackage{color}
\usepackage{graphicx}
\usepackage{enumitem}
\usepackage[bookmarksnumbered,colorlinks, linkcolor=blue, citecolor=red, pagebackref, bookmarks, breaklinks]{hyperref}

\newtheorem{thm}{Theorem}[section]
\newtheorem{cor}[thm]{Corollary}
\newtheorem{lema}[thm]{Lemma}
\newtheorem{prop}[thm]{Proposition}
\theoremstyle{definition}
\newtheorem{defn}[thm]{Definition}
\theoremstyle{remark}

\newtheorem{rem}[thm]{Remark}
\numberwithin{equation}{section}
\newcommand{\R}{\mathbb R}
\newcommand{\N}{\mathbb N}

\newcommand{\s}{\mathbf{s}}
\newcommand{\p}{\mathbf{p}}
\newcommand{\Wsp}{W^{\s,\p}}

\def\supp{\mathop{\text{\normalfont supp}}}

\newcommand{\intRn}{\int_{\R^n}}
\newcommand{\intR}{\int_{\R}}
\newcommand{\subscript}[2]{$#1 _ #2$}

\begin{document}

\title[A BBM formula for anisotropic fractional Sobolev spaces]{A Bourgain-Brezis-Mironescu formula for anisotropic fractional Sobolev spaces and applications to anisotropic fractional differential equations}

\author[I. Ceresa Dussel and J. Fern\'andez Bonder]{Ignacio Ceresa Dussel and Juli\'an Fern\'andez Bonder}

\address{Instituto de C\'alculo, CONICET\\
Departamento de Matem\'atica, FCEN - Universidad de Buenos Aires\\
Ciudad Universitaria, 0+$\infty$ building, C1428EGA, Av. Cantilo s/n\\
Buenos Aires, Argentina}

\begin{abstract}
In this paper we prove Bourgain-Brezis-Mironescu's type results (cf. \cite{BBM2001}) (BBM for short) for an energy functional which is strongly related to the fractional anisotropic p-Laplacian.  We also provide with the analogous of Maz'ya-Shaposhnikova (see \cite{MS}) type results for these energies and finally we apply these results to analyze the stability of solutions to anisotropic fractional $p-$laplacian equations.
\end{abstract}


%46E35  	Sobolev spaces and other spaces of "smooth'' functions, embedding theorems, trace theorems
%35J92  	Quasilinear elliptic equations with $p$-Laplacian
%47G10  	Integral operators
%35R11  	Fractional partial differential equations
%35B27  	Homogenization in context of PDEs; PDEs in media with periodic structure


\maketitle

\section{Introduction}
The study of fractional order Sobolev spaces its a classical topic in analysis that has drawn the attention of the mathematical community at least since the middle of the XX century. See, for instance the classical book of Stein \cite{Stein}.

In the last decades, there has been an increasing interests in this topic since new applications of these spaces appear to some models in different fields of the natural sciences such as physical models \cite{DGLZ, Eringen, Giacomin-Lebowitz, Laskin, Metzler-Klafter, Zhou-Du}, finance \cite{Akgiray-Booth, Levendorski, Schoutens}, fluid dynamics \cite{Constantin, Dalibard}, ecology \cite{Humphries, Massaccesi-Valdinoci, Reynolds-Rhodes} and image processing \cite{Gilboa-Osher}.

In particular, a relevant problem is the transition between nonlocal models and local ones. This, for instance, is the transition of the fractional space $W^{s,p}$ to the classical Sobolev space $W^{1,p}$ as the fractional parameter $s\uparrow 1$ and also the case where the fractional parameter $s\downarrow 0$ where the space $W^{s,p}$ approaches the Lebesgue space $L^p$. These problems were analyzed in the seminal papers \cite{BBM2001} (for $s\uparrow 1$) and \cite{MS} (for $s\downarrow 0$).

An interesting application of the results in \cite{BBM2001} is the asymptotic behavior of solutions to the fractional laplacian (or $p-$laplacian) as $s\uparrow 1$. That is, if we consider the problem
\begin{equation}\label{p-lap.intro}
\begin{cases}
(-\Delta_p)^s u = f & \text{in }\Omega\\
u=0 & \text{in } \R^n\setminus \Omega,
\end{cases}
\end{equation}
where $(-\Delta_p)^s u$ is the fractional $p-$Laplace operator defined as
$$
(-\Delta_p)^s u(x). = \text{p.v.} (1-s)K_{n,p}\int_{\R^n} \frac{|u(x)-u(y)|^{p-2}(u(x)-u(y))}{|x-y|^{n+sp}}\, dy,
$$
and $u_s$ is a solution to \eqref{p-lap.intro}, then one expect that as $s\uparrow 1$, $u_s$ converge to a solution of the local problem
$$
\begin{cases}
-\Delta_p u = f & \text{in }\Omega\\
u=0 & \text{on }\partial\Omega. 
\end{cases}
$$
This problem was analyzed, for instance, in \cite{BH, BS, BPS, FBS, FBSS}, and was found to be true in several cases, for instance if $f=f(x)\in L^{p'}(\Omega)$ and if $f=f(x,u)$ under some growth control on $f$.

\medskip

On the other hand, it is interesting to analyze Sobolev spaces and the corresponding operators in the case where some anisotropy is present. This study was carried on in \cite{CKM} where the authors analyze the fractional anisotropic Sobolev space, that in each coordinate direction the function has different fractional regularity and different integrability. That is, in \cite{CKM} the authors consider $n$ different fractional parameters $s_1,\dots, s_n$ and $n$ different integrability parameters $p_1,\dots,p_n$ and define the space of functions $u(x)=u(x_1,\dots,x_n)$ such that
$$
\intRn\intR\frac{|u(x+he_i)-u(x)|^{p_i}}{|h|^{1+s_i p_i}}\, dh dx<\infty, \quad \text{for } i=1,\dots,n,
$$
where $e_i = (0,\dots,\underbrace{1}_{i},\dots,0)$.

In the above mentioned paper \cite{CKM} the authors study these fractional anisotropic Sobolev spaces, prove some Sobolev-type immersion theorems and the existence of extremals in $\R^n$ in the critical case.

\medskip

In view of the above discussion, the purpose of this paper is  to extend the classical results of \cite{BBM2001} and \cite{MS} to the anisotropic case and then apply these extensions to the study of the asymptotic behavior of solutions to anisotropic fractional $p-$Laplace equations when (some) fractional parameters approaches $1$.

\subsection*{Organization of the paper}
The paper is organized as follows. In Section \ref{S2}, we introduce the necessary preliminaries needed in this work, and discuss in detail some of the results in \cite{CKM} that are related to our paper. Moreover we prove some important estimates on anisotropic fractional energies that will play a key role in the rest of the article (see for instance Lemma \ref{lemma5}). In Section \ref{S3} we prove the extension of the classical Bourgain-Brezis-Mironescu's result \cite{BBM2001} to the anisotropic case, see Theorems \ref{BBM1} and \ref{sequence case}. In Section \ref{S4} we show the extension of Maz'ya-Shaposhnikova's result to our case, Theorem \ref{ThmMS}, and finally in Section \ref{S5} we analyze the applications of the above mentioned results to the asymptotic behavior of solutions to fractional anisotropic $p-$Laplace problems when the fractional parameters (or some of them) goes to 1. We analyze the cases where the source term $f$ is a fixed function of the space variable $x$ (Corollary \ref{fixed f}), then the general case where the source term is a nonlinearity of the form $f=f(x,u)$ (Theorem \ref{stability.semilinal}) and finally, we study the case of ground state solution proving that the limit of ground state solutions to fractional problems converge to the corresponding ground state solutions of the limit problems (Theorem \ref{stability.ground}).

\section{Preliminaries}\label{S2}

Let us start by defining the anisotropic spaces that will be considered in this work.

Given $i=1,\dots, n$,  $s\in(0,1]$ and $p\in (1,\infty)$ we define the fractional $i^{th}-$Sobolev space as 
$$
W_i^{s,p}(\R^n)=\{u \in L^p(\R^n) \text{ such that } J^i_{s,p}(u)< \infty\},
$$
where 
 \[
  J^i_{s,p}(u) = 
  \left \{
    \begin{aligned}
      &(1-s)s\intRn\intR\frac{|u(x+he_i)-u(x)|^p}{|h|^{1+sp}}\, dh dx\ &\text{ if } \ &s < 1\\
      &\frac{2}{p}\intRn |\partial_{x_i} u|^p \, dx\ &\text{ if } \ &s=1.
    \end{aligned}
  \right .
\]

It is easy to see that $W^{s,p}_i(\R^n)$ is a Banach space with norm
$$
\|u\|_{i,s,p} = \|u\|_p + (J^i_{s,p}(u))^{1/p},
$$
which is separable and reflexive.

This space consists on $L^p$ functions that has some mild regularity in the $i^{th}-$variable. In order to control every variable, we proceed as follows. Let $\s=(s_1,\dots, s_n)$ and $\p=(p_1,\dots, p_n)$, with $s_i\in (0,1]$ and $p_i\in (1,\infty)$ for $i=1,\dots, n$. Then define
$$
W^{\s,\p}(\R^n)=\bigcap_{i=1}^n W^{s_i,p_i}_i(\R^n).
$$
This anisotropic Sobolev space is then a separable and reflexive Banach space with norm
$$
\|u\|_{\s,\p} = \sum_{i=1}^n \|u\|_{i,s_i,p_i}.
$$

In most applications, what turns out to be more useful than the norm is the functional $J_{\s,\p}$ that is defined as
$$
J_{\s,\p}(u) = \sum_{i=1}^n \frac{1}{p_i} J^i_{s,p}(u)
$$ 
for $u \in W^{\s,\p}(\R^n)$.

If we consider $\Omega \subset \R^n$, then we define the anisotropic Sobolev space of functions vanishing at the boundary $\partial\Omega$ as
$$
W^{\s,\p}_0(\Omega)=\{u \in W^{\s,\p}(\R^n) \colon u|_{\Omega^c}=0\}.
$$

\begin{rem}
Recall that this is not the most common definition of functions vanishing at the boundary. An alternative definition is
\begin{equation}\label{W0}
\widetilde{W}^{\s,\p}_0(\Omega) = \overline{C^\infty_c(\Omega)},
\end{equation}
where the closure is taken with respect to the norm $\|\cdot\|_{\s,\p}$.

In the classical case, where $s= s_1 = s_2 = \dots = s_n$ and $p= p_1=p_2=\dots=p_n$, it is proved in \cite{DPV} that both spaces agree, $\widetilde{W}^{\s,\p}_0(\Omega) = W^{\s,\p}_0(\Omega)$, if for instance, $\Omega$ is Lipschitz.

We don't investigate this issue here for the anisotropic case, and will consider the definition given in \eqref{W0}.
\end{rem}

Observe that if we define
$$
p_\text{min} = \min\{p_1,\dots,p_n\}\qquad \text{and} \qquad p_\text{max} = \max\{p_1,\dots,p_n\}
$$
then, by standard interpolation we have that
$$
W^{\s,\p}(\R^n) = \left\{u\in L^{p_\text{min}}(\R^n)\cap L^{p_\text{max}}(\R^n) \colon J_{\s,\p}(u)<\infty \right\}.
$$

In the {\em local} anisotropic case, i.e. when $s_1=s_2=\dots=s_n=1$, it is proven in  \cite{EH-R-2007} that the space $W^{1,\p}(\R^n)$ is continuously embedded in $L^{\p^*} (\R^n)$, where $\p^*$ is the associated critical exponent given by
$$
\p^*=\frac{n}{\sum_{i=1}^n p_i^{-1}-1}.
$$

On the other hand, in the general fractional anisotropic case, Sobolev immersion theorems for these spaces was studied in \cite{CKM}. In that paper, the authors define the quantities
$$
\overline{\s}=\left(\frac{1}{n}\sum_{i=1}^n\frac{1}{s_i}\right)^{-1},  \qquad  \overline{\s\p}=\left(\frac{1}{n}\sum_{i=1}^n\frac{1}{s_ip_i}\right)^{-1}.
$$
Then, assuming that $\overline{\s\p} < n$ the following critical Sobolev exponent is introduced
\begin{equation}\label{p*}
\p^*_\s = \p^*=\frac{n\overline{\s\p}/\overline{\s}}{n-\overline{\s\p}}.
\end{equation}
Under the further assumption that $p_\text{max} < \p^*$ the following Sobolev-Poincar\'e inequality is proved in \cite{CKM}:
$$
0<S=\inf_{u\in W^{\s,\p}(\R^n) \atop \|u\|_{\p^*}=1} J_{\s,\p}(u).
$$
where the constant $S$ depends on $n, \p^*, \p^*-p_\text{max}$ and on a lower bound $s_0$ of $s_i$ for $i=1,\dots,n$.

Moreover, a compact-embedding theorem, $\Wsp(\R^n)\subset \subset L^q_{loc}(\R^n),$ for every $1\leq q< \p^*$ is also proved in \cite{CKM}.

\bigskip

In the remaining of the section we prove some results on the functionals $J^i_{s,p}$ that will be helpful in the sequel. When dealing with a single functional, without loss of generality we can assume that $i=1$.

\begin{lema}\label{lemma1}
 Let $u$ be a function in $W^{1,p}_1(\R^n)$, then  
\begin{eqnarray}
J^1_{s,p}(u)\leq \frac{1}{p}\left(s\left\|\partial_{x_1}u\right\|_{p}^{p}+ (1-s) 2^{p+1}\|u\|_{p}^{p}\right).\nonumber
\end{eqnarray}
\end{lema}

\begin{proof}
Given $u \in C^1_c(\R^n)$ it is a well known fact that
$$\intRn|u(x+he_1)-u(x)|^{p}\, dx\leq |h|^{p} \|{\partial_{x_1}}u\|_{p}^{p},
$$ 
so
\begin{align*}
\intR\intRn\frac{|u(x+he_1)-u(x)|^{p}}{|h|^{1+sp}}\, dx dh =& \int_{-1}^1\frac{1}{|h|^{1+sp}}\left(\intRn|u(x+he_1)-u(x)|^{p}\, dx\right)\, dh\ \\
& + \int_{|h|>1}\frac{1}{|h|^{1+sp}}\left(\intRn|u(x+he_1)-u(x)|^{p}\, dx\right)\, dh\\
=& I + II.
\end{align*}
For the first term,
\begin{equation}\label{primeras desigualdades}
I\leq \left\|\partial_{x_1} u\right\|_{p}^{p}2\int_{0}^1\frac{h^p}{|h|^{1+sp}}\, dh=\frac{1}{p(1-s)}\left\|\partial_{x_1} u\right\|_{p}^{p}
\end{equation}
and for the second term,
\begin{equation}\label{primeras desigualdades2}
II\leq 2^{p+1}\|u\|_{p}^{p}\int_1^{\infty}\frac{1}{|h|^{1+sp}}dh=2^{p+1}||u||_{p}^{p}\frac{1}{sp}.
\end{equation}
Now combining \eqref{primeras desigualdades} and \eqref{primeras desigualdades2} we conclude the proof of the estimate for smooth functions. The result is finished just observing that smooth functions with compact support are dense in $W^{1,p}_1(\R^n)$.
\end{proof}

The next lemma shows that the functional $J^1_{s,p}$ decreases when one regularizes a function $u$ with a standard mollifier.
\begin{lema}\label{lemma2}
 Let $u \in L^p(\R^n)$ and $u_\epsilon=\rho_\epsilon * u $, where $\rho_\epsilon$ is the standard mollifier. Then 
 $$
J^1_{s,p}(u_\epsilon)\le J^1_{s,p}(u)
$$
for every $ \epsilon >0$.
\end{lema}

\begin{proof}

\begin{align*}
\intRn |u_{\epsilon}(x+he_1)-u_{\epsilon}(h)|^{p} dx &= \intRn\left|\intRn\left(u(x+he_1-y)-u(x-y)\right)\rho(y)dy\right|^{p}\, dx\\
&\leq \intRn\intRn\left|u(x+he_1-y)-u(x-y)\right|^{p} \rho(y)\, dy dx\\
&= \intRn\left(\intRn\left|u(x+he_1-y)-u(x-y)dx\right|^{p}\, dx\right) \rho(y)\, dy \\
&= \intRn|u(x+he_1)-u(x)|^{p} dx,
\end{align*}
where we have used the invariance of the norm with respect to translations and that $\rho$
has integral equal to 1. Finally, multiplying by $|h|^{-1-sp}$ and integrating over $\R$ we get the desired estimate.
\end{proof}

Now, we prove a lemma that shows the behavior of $J^1_{s,p}$  with respect to truncations.
\begin{lema}\label{lemma3}
Let $\eta\in C^1_c(\R^n)$ be such that $\eta =1$ on $B_1$, $\eta=0$ on $B_2^c$, $0\le \eta\le 1$ and $\|\nabla \eta\|_\infty\le 2$. Given $k\in\N$, we define $\eta_k(x)=\eta(x/k)$.

 Let $u \in L^p(\R^n)$ and $u_k=\eta_ku$ be the truncation of $u$ by $\eta_k$ then
 $$
 J^1_{s,p}(u_k) \leq C(n,p) \left( J^1_{s,p}(u) + \|u\|_p^p \right),
 $$
 where $C(n,p)>0$ is independent from k.
\end{lema}

\begin{proof} 
First, observe that we have
$$
|u_k(x+he_1)-u_k(x)|^{p}\leq 2^{p-1}(|\eta_k(x+he_1)|^p |u(x+he_1)-u(x)|^p + |u(x)|^p |\eta_k(x+he_1)-\eta_k(x)|^p).
$$
Hence, we get
\begin{align*}
\intRn\intR\frac{|u_{k}(x+he_1)-u_{k}(x)|^{p}}{|h|^{1+sp}}\, dh dx \leq &2^{p-1} \intRn\intR\frac{|u(x+he_1)-u(x)|^{p}}{|h|^{1+sp}}\, dh dx +\nonumber \\ 
& 2^{p-1} \intRn |u(x)|^{p} \intR\frac{|\eta_k(x+he_1)-\eta_k(x)|^{p}}{h^{1+sp}}\, dh\, dx.
\end{align*}

To finish the proof, we observe that
\begin{align*}
\intR \frac{|\eta_k(x+he_1)-\eta_k(x)|^p}{h^{1+sp}}\, dh&\leq \left(\int_{|h|<1}+\int_{|h|\ge 1}\right)  \frac{|\eta_k(x+he_1)-\eta_k(x)|^{p}}{h^{1+sp}}\, dh\\
&= I + II.
\end{align*}
Now,
$$
I \le \|\partial_{x_1} \eta_k\|_\infty^p \int_{|h|<1} \frac{1}{|h|^{1+sp-p}}\, dh = \frac{2^{p+1}}{k^p} \frac{1}{p(1-s)}.
$$
Finally,
$$
II \le 2^{p-1} \|\eta_k\|_\infty^p \int_{|h|\ge 1} \frac{1}{|h|^{1+sp}}\, dh = \frac{2^p}{sp}.
$$
and our lemma has been proved.

\end{proof}

The next lemma is proved in \cite{BBM2001} and is crucial in the arguments that follows.
\begin{lema}\label{lemma4}
Let $\phi,\psi\colon (0,1)\to \R$ be measurables functions and suppose that $\phi(2h)\le \phi(h)$ for $h\in (0,1)$ and that $\psi$ is nonincreasing.Then given  $r>-1,$
$$
\int_0^1h^r \phi(h) \psi(h)\, dh \ge \frac{r+1}{2^{r+1}}\int_0^1 h^r \phi(h)\, dh \int_0^1 h^r \psi(h)\, dh.
$$
\end{lema}

With the help of Lemma \ref{lemma4} we can prove the following immersion result.
\begin{lema}\label{lemma5}
Let $1<p<\infty$ and $0<s_1<s_2<1$  then
\begin{align*}
J^1_{s_1,p}(u) \leq 2^{(1-s_1)p}J^1_{s_2,p}(u) + \frac{(1-s_1)2^{p+1}}{p}\|u\|_p^p ,
\end{align*}
for all $u\in L^p(\R^n)$.
\end{lema}

\begin{proof}
First, let us begin by observing that
$$
\intR \intRn \frac{|u(x+he_1)-u(x)|^p}{|h|^{1+sp}}\, dxdh = 2\int_0^\infty \intRn \frac{|u(x+he_1)-u(x)|^p}{|h|^{1+sp}}\, dxdh.
$$
Hence, the lemma will be proved if we show that
\begin{align*}
(1-s_1)\int_0^\infty \intRn&\frac{|u(x+he_1)-u(x)|^p}{|h|^{1+s_1p}}dxdh\\
\leq& 2^{(1-s_1)p}(1-s_2)\int_0^\infty \intRn\frac{|u(x+he_1)-u(x)|^p}{|h|^{1+s_2p}}dxdh +\frac{(1-s_1)2^p}{s_1p}\|u\|_p^p ,
\end{align*}

Define $F$ as 
$$
F(h)=\intRn|u(x+he_1)-u(x)|^p \, dx,
$$

since
\begin{align*}
F(2h)&\leq 2^{p-1}\left(\intRn|u(x+2he_1)-u(x+he_1)|^p\, dxdh\right)\\ &+2^{p-1}\left(\intRn|u(x+he_1)-u(x)|^p\, dxdh\right)\\
&= 2^p\intRn|u(x+he_1)-u(x)|^p\, dxdh =2^p F(h) ,
\end{align*}
We can take $g(h)=F(h)/h^p$ and get that $g(2h)\leq g(h)$ for all $h>0$.

On the other hand, 
\begin{equation}\label{igualdad_de_func_decre}
\begin{split}
    \int_0^1\intRn\frac{|u(x+he_1)-u(x)|^p}{|h|^{1+sp}}\, dxdh &= \int_0^1 \frac{F(h)}{|h|^{1+sp}}\,dh\\
&=\int_0^1 \frac{g(h)}{|h|^{1-(1-s)p}}\,dh
\end{split}
\end{equation}
Now, let us consider $0<s_1<s_2<1$, then
$$\int_0^1 \frac{g(h)}{|h|^{1-(1-s_2)p}}\,dh=\int_0^1 \frac{1}{|h|^{1-(1-s_1)p}}g(h)\frac{1}{|h|^{(s_2-s_1)p}}\,dh$$
and applying Lemma \ref{lemma4} with $r=(1-s_1)p-1$ and $\psi(h)=h^{-(s_2-s_1)p}$, 
we obtain 
\begin{align*}
    \int_0^1 \frac{g(h)}{|h|^{1-(1-s_2)p}}\,dh &\geq\frac{(1-s_1)p}{2^{(1-s_1)p}}\int_0^1\frac{g(h)}{h^{1-(1-s_1)p}}\,dh\int_0^1\frac{1}{h^{1-(1-s_2)p}}dh\\&=\frac{1}{2^{(1-s_1)p}}\frac{1-s_1}{1-s_2}\int_0^1\frac{g(h)}{h^{1-(1-s_1)p}}\,dh.
\end{align*}
Recalling \eqref{igualdad_de_func_decre} 
\begin{align*}
    \frac{1-s_1}{2^{(1-s_1)p}}&\int_0^1\int_{R^n}\frac{|u(x+he_1)-u(x)|^p}{|h|^{1+s_1p}}\,dxdh\\&\leq (1-s_2)\int_0^1\intRn\frac{|u(x+he_1)-u(x)|^p}{|h|^{1+s_2p}}\,dxdh.
\end{align*}
Finally,
$$
\int_1^\infty\intRn \frac{|u(x+he_1)-u(x)|^p}{|h|^{1+s_1p}}\,dxdh\leq\frac{2^p\|u\|^p_p}{s_1p}
$$
this completes the proof.
\end{proof}

\section{BBM-type results for $W^{s,p}_i$}\label{S3}
In this section we extend the results in the seminal paper \cite{BBM2001} to the anisotropic case.
Our first result conserns with the asymptotic behavior of the funcional $J^1_{s,p}$ as $s\uparrow 1$ in a pointwise sense.

\begin{thm}\label{BBM1}
Let $u \in L^p(\R^n)$, $s\in(0,1)$ then
$$
\lim_{s\to 1}J^1_{s,p}(u) = J^1_{1,p}(u).
$$

\end{thm}
\begin{proof}
To prove the theorem \ref{BBM1} we split the proof in 3 steps.

\medskip

\textbf{Step 1:} Let $u \in C^{2}_0(\R^n)$ , $s \in (0,1)$ and $p \in (1,\infty)$. Arguing as in \cite[Theorem 2]{BBM2001} we easily obtain that
$$
\left|\frac{|u(x+he_1)-u(x)|^{p}}{|h|^p}-\left|\partial_{x_i}u\right|^p\right|\leq C|h|,
$$
where $C$ depends on the $C^2-$norm of $u$ and $p$.

Naming 
$$
A=\frac{|u(x+he_1)-u(x)|^{p}}{|h|^p}-\left|\partial_{x_i}u\right|^p
$$
 and taking the integral over $\R$, 

\begin{align*}
    \int_\R A\,dh=\int_{|h|\leq 1}A\,dh+\int_{|h|>1}A\,dh,
\end{align*}
where by routine integration,
$$
\int_{|h|\leq 1}A\,dx=2C\frac{1}{1+p(1-s)} \qquad\text{and}\qquad \int_{|h|> 1}A\,dx=2^p\frac{\|u\|_\infty^p}{ps}.
$$
Finally, taking limit $s\to 1$ we conclude that 
\begin{equation}\label{eq1}
\lim_{s\to 1}s(1-s)\int _\R \frac{|u(x+he_1)-u(x)|^{p}}{|h|^{1+sp}}dh=\frac{2}{p}|\partial_{x_1}u|^p.
\end{equation}

Now, the key is look up an integrable majorant
 for \eqref{eq1}.
 
 Let 
$$
U_s(x)=\int_\R\frac{|u(x+he_1)-u(x)|^p}{|h|^{1+sp}}\,dh.
$$
Then since $u\in  C^2_0(\R^n) $, there exists $R>0$ such that $\supp(u)\subset Q_R$, where $Q_R=[-R,R]^n$. 

Thus, we split the problem:
 \begin{itemize}
        \item If $|x_1|\leq 2R$
        \begin{equation}
            \begin{split}
            |U_s(x)|&=\int_\R\frac{|u(x+he_1)-u(x)|^p}{|h|^{1+sp}}\,dh \\&= \left(\int_{|h|\leq 3R}+\int_{|h|>3R}\right)\frac{|u(x+he_1)-u(x)|^p}{|h|^{1+sp}}\,dh 
            \end{split},\nonumber
        \end{equation}
        where we bound the first integral by
            \subitem
        \begin{equation}
            \int_{|h|\leq 3R}\frac{|u(x+he_1)-u(x)|^p}{|h|^{1+sp}}\,dh\leq ||\partial_{x_1}u||_{\infty}^p \frac{3R^{p}}{p(1-s)}\nonumber,
        \end{equation}
            \subitem
            and the second by
         \begin{equation}
            \int_{|h|>3R}\frac{|u(x+he_1)-u(x)|^p}{|h|^{1+sp}}\,dh\leq||u||_{\infty}^p\frac{3R^{-p(1/2)}}{(1-s)p}.\nonumber
         \end{equation}
        
         Then 
         $$
         |s(1-s)U_s(x)|\leq C \chi_{Q_{2R}}(x),
         $$
          where  C is a constant depending of $p,R$ but not on $s$.
        \item If $|x_1|> 2R$, 
        \begin{equation}
            \begin{split}
            |U_s(x)|=\int_\R\frac{|u(x_1+h,x')|^p}{|h|^{1+sp}}dh\leq \left(\frac{2}{|x_1|}\right)^{1+sp}\int_\R|u(x_1,x')|^p dx_1. 
        \end{split}\nonumber
        \end{equation}
           \end{itemize}
     Where we have used that
     $ |h|\geq |x_1|-|x_1+h|\geq |x_1|-R\geq \frac{1}{2}|x_1|$.

Observe that     
     \begin{equation}
         G(x'):=\int_\R|u(x_1,x')|^p dx_1 \in L^1(\R^{n-1})  \nonumber
     \end{equation}
     Therefore,
     $$
     |s(1-s)U_s(x)|\leq C\left( \chi_{Q_{2R}}(x)+\chi_{Q_{2R}^c}(x)\left(\frac{2}{|x_1|}\right)^{1+(1/2)p}G(x')\right)\in L^1(\R^n).
     $$
 To conclude this step, we only have to use the Lebesgue's Dominated Convergence Theorem and obtain 
 \begin{equation}
     \lim_{s\to 1}s(1-s)\intRn\intR\frac{|u(x+he_1)-u(x)|^{p}}{|h|^{1+sp}}\,dhdx=
 \frac{2}{p} \intRn\left|\partial_{x_1} u\right|^{p}\,dx,
\end{equation}
for all $u \in C^2_0(\R^n).$

\medskip

\textbf{Step 2:} Let $u \in W^{1,p}_1(\R^n)$, $s \in (0,1)$ and $p \in (1,\infty) $ . Take a sequence $\{u_k\}_{k\in \N} \subset C^2_0(\R^n)$ such that $u_k\to u$ in $L^p(\R^n)$ and $\partial_{x_1} u_k\to \partial_{x_1} u$ in $L^p(\R^n)$.
 Let
    \begin{equation}\label{seminorm}
        [u]_{1,s,p}= (J^1_{s,p}(u))^{1/p}, \qquad 0<s\leq 1.
    \end{equation}
    Observe that $[\, \cdot\, ]_{1,s,p}$ defined in \eqref{seminorm} is a seminorm and thus satisfies the triangle inequality. So
    
\begin{align*}
\left| [u]_{1,s,p} - [u]_{1,1,p}\right| & \leq    | [u]_{1,s,p} - [u_k]_{1,s,p}| + |[u_k]_{1,s,p}-[u_k]_{1,1,p}| + |[u_k]_{1,1,p} - [u]_{1,1,p}|\\
         &= I  + II  + III.
        \end{align*}
        
         Using lemma \ref{lemma1}, 
       $$
             I\leq|[u_k-u]_{1,s,p}|\leq C(n,p) \left( \|u_k-u\|_{p} + \|\partial_{x_1} u_k - \partial_{x_1} u\|_p\right).
         $$
         which converges to zero if $k\to \infty$, uniformly in s. The third term is bounded in much the same way.  

Finally, the second term is consequence of Step 1, and so the proof of Step 2 is complete.
    
\medskip

\textbf{Step 3:} Given $u \in L^{p}(\R^n)$ and assume that
\begin{equation}\label{liminf1}
    \liminf_{s \to 1}J^1_{s,p}(u) < \infty,
\end{equation}
 
Let $u_{k,\epsilon}\subset C^\infty_0(\R^n)$ be given by
$$
u_{k,\epsilon}=\rho_\epsilon*(u\eta_k)
$$
and  applying lemmas \ref{lemma2} and \ref{lemma3} together with \eqref{liminf1}, we obtain

$$
\liminf_{s\to1} J^1_{s,p}(u_{k,\epsilon}) \le C,
$$
where $C$ is independent from $k$ and $\epsilon>0$.

Next, for any fixed $k$ and $\epsilon$, we apply Step 1 and obtain
\begin{equation}
C\geq \lim_{s\to 1} J^1_{s,p}(u_{k,\epsilon}) = J^1_{1,p}(u_{k,\epsilon}).\nonumber
\end{equation}

Hence, there exist a subsequence $u_j=\{u_{k_j,\epsilon_j}\}\subset \{u_{k,\epsilon}\}$ such that $\partial_{x_1} u_j \rightharpoonup \partial_{x_1} u$ weakly in $L^p(\R^n)$. 
Thus $u \in W_1^{1,p}(\R^n)$. and so the proof of Step 3 follows from Step 2.

This completes the proof of the Theorem.
\end{proof}

Next, we will consider the case that, instead of having a fixed function $u$ for every $s$, the
function also depends on the fractional parameter $s$.
\begin{thm}\label{sequence case}
Let  $0<s_k\to 1$ when $k\to \infty$. Assume that $\{u_k\}\subset L^{p}(\R^n) $ is such that
 \begin{equation}\label{hyp-teo}
\sup_{k}\|u_k\|_{p}<\infty,  \quad u_k\to u \text{ in } L^p_\text{loc}(\R^n) \quad \text{and}\qquad \sup_{k\in \mathbb{N}} J^1_{s_k, p}(u_k)<\infty.
 \end{equation}
 Then,
 \begin{equation}\label{liminf}
 J^1_{1,p}(u)  \leq
     \liminf_{s_k\to1} J^1_{s_k,p}(u_k).
 \end{equation}
\end{thm}
\begin{proof}

 Let $\{u_k\}\subset L^p(\R^n)$ be such that \eqref{hyp-teo} holds. Let now $t\in (0,1)$ be fixed but arbitrary. By the $L^p_\text{loc}(\R^n)$ convergence of the sequence $\{u_k\}$, we can assume that $u_k \to u$ a.e in $\R^n.$

By Fatou's lemma,
\begin{equation}
J^1_{t,p}(u) \leq \liminf_{k\to\infty} J^1_{t,p}(u_k) \nonumber
\end{equation}
Applying now Lemma  \ref{lemma5}, with $M = \sup_k\|u_k\|_p$ we get
\begin{align*}
J^1_{t,p}(u)\le 2^{(1-t)p}\liminf_k J^1_{s_k,p}(u_k) + (1-t)\frac{2^{p+1}}{p} M^p,
\end{align*}
and the result follows taking  $t \to 1$ and using Theorem \ref{BBM1}.
\end{proof}

Another useful result is the following compactness result for a sequence $u_k\in W^{\s_k,\p}(\R^n)$.
\begin{thm}\label{compactness}
Let $\p=(p_1,\dots,p_n)$ with $p_i\in (1,\infty)$ for every $i=1,\dots, n$. Assume that $s_i^k\to 1$ as $k\to\infty$ for $i=1,\dots,j$ and let $\s_k=(s_1^k,\dots,s_j^k,s_{j+1},\dots,s_n)$. Then $\s_k\to \s_0 = (\underbrace{1,\dots,1}_{j}, s_{j+1},\dots,s_n)$. 

Let $\{u_k\}\subset L^{p_{\text{min}}}(\R^n) \cap  L^{p_{\text{max}}}(\R^n)$ be a sequence such that
$$
\sup_{k\in\N} \|u_k\|_{p_\text{max}}+\|u_k\|_{p_\text{min}}<\infty \qquad\sup_{k\in \mathbb{N}} J_{{\s_k},\p}(u_k)<\infty.
$$
Then, if $p_\text{max} < \p^*_{\s_0}$, there exist a subsequence $\{u_{k_j}\}_{j\in\N}\subset \{u_k\}_{k\in\N}$  such that $u_{k_j}\to u$ as $j\to\infty$ in $L^{p_\text{max}}_\text{loc}(\R^n)$.
 \end{thm}

\begin{proof}
Since $\s_k\to\s_0$, it follows that $\p^*_{\s_k}\to \p^*_{\s_0}$. Therefore, we can assume that $p_\text{max} < \p^*_{\s_k}$.

Now, take $t\in (0,1)$ close to 1, such that it $\mathbf t = (\underbrace{t,\dots,t}_{j}, s_{j+1},\dots,s_n)$, then $p_\text{max}<\p^*_{\mathbf t}$.

Applying now Lemma \ref{lemma5} we obtain that
$$
J_{\mathbf t, \p}(u_k)\le C (J_{\s_k,\p}(u_k) + 1) \le C\quad \text{for every } k\in\N.
$$

Hence, we can apply the Rellich-Kondrashov type result for the space $W^{\mathbf t,\p}(\R^n)$ \cite[Theorem 2.1]{CKM} and hence the convergence of $\{u_k\}_{k\in\N}$ (up to a subsequence) in $L^{p_\text{max}}_\text{loc}(\R^n)$ is guaranteed. 
\end{proof}

To conclude the section, let us now make the connection of  Theorems \ref{BBM1} and \ref{sequence case} with the notion of Gamma-convergence.

Recall that in a metric space $(X,d)$, a sequence of functions $f_k\colon X\to \bar\R$ is said to Gamma converges to a limit function $f$ it this two conditions are satisfied:
\begin{itemize}
\item ($\liminf$ inequalitty) For every $x\in X$ and every sequence $\{x_k\}_{k\in\N}\subset X$ such that $x_k\to x$, it holds that
$$
f(x)\le \liminf_{k\to\infty} f_k(x_k).
$$
\item ($\limsup$ inequality) For every $x\in X$ there exists a sequence (called the {\em recovery sequence}) $\{y_k\}_{k\in \N}\subset X$ such that
$$
f(x)\ge \limsup_{k\to \infty} f_k(y_k).
$$
\end{itemize}
The notion of Gamma convergence is a useful tool in the Calculus of Variations and was introduced by De Giorgi in the 70's. See \cite{DalMaso} for a throughout introduction to the subject.

Observe that Theorems \ref{BBM1} and \ref{sequence case} immediately imply that $J^1_{s,p}$ Gamma converges to $J^1_{1,p}$ as $s\uparrow 1$ in $L^p(\Omega)$ for every $\Omega\subset \R^n$ bounded. We collect this statement in the following corollary.
\begin{cor}\label{Gamma1}
Let $s_k\to 1$ and $p\in (1,\infty)$. Then $J^1_{s_k,p}$ Gamma converges to $J^1_{1,p}$ as $k\to\infty$ in $L^p_\text{loc}(\R^n)$.
\end{cor}
\begin{proof}
Just observe that Theorem \ref{sequence case} is exactly the $\liminf$ inequality and for the $\limsup$ inequality one just have to take the constant sequence $u_k=u$ as a recovery sequence and apply Theorem \ref{BBM1}.
\end{proof}

In the forthcoming application (cf. Section \ref{S5}) it will be important to have the Gamma convergence of the full functional
$$
J_{\s,\p}(u) = \sum_{i=1}^n \frac{1}{p_i} J^i_{s_i,p_i}(u).
$$
Moreover we consider a sequence $\s_k = (s_1^k, s_2^k,\dots, s_n^k)$ where for instance $s_1^k,\dots,s_j^k\to 1$ as $k\to\infty$ and $s_{j+1}^k,\dots,s_n^k$ remains constant.

Hence $\s_k\to \s_0 = (\underbrace{1,\dots,1}_{j}, s_{j+1},\dots,s_n)$.

In principle, it is not true that sums of Gamma convergent function is Gamma convergent (see \cite{DalMaso}). However, in this case since the recovery sequence is the constant sequence (and in particular is the same for every function), the Gamma convergence of $J_{\s_k,\p}$ holds true.

Let us collect these results in the next corollary.
\begin{cor}\label{Gamma2}
Let $\p=(p_1,\dots,p_n)$ with $p_i\in (1,\infty)$ for every $i=1,\dots, n$. Let $s_i^k\to 1$ as $k\to\infty$ for $i=1,\dots,j$ and let $\s_k=(s_1^k,\dots,s_j^k,s_{j+1},\dots,s_n)$. Let $\s_0 = (\underbrace{1,\dots,1}_{j}, s_{j+1},\dots,s_n)$. Then $J_{\s_k,\p}$ Gamma converges to $J_{\s_0,\p}$ in $L^{p_\text{max}}_\text{loc}(\R^n)$.
\end{cor}
\begin{proof}
The $\liminf$ inequality for $J_{\s_k,\p}$ follows from the $\liminf$ inequality for every $J^i_{s_i^k,p_i}$ $(i=1,\dots,j)$ and the continuity of $J^i_{s_i,p_i}$ for $i=j+1,\dots,n$.

The $\limsup$ inequality follows since the constant sequence $u_k=u$ is a recovery sequence for {\em every} $J^i_{s_i^k,p_i}$, $i=1,\dots,j$.
\end{proof}

\section{Mazya-Shaposhnikova--type results}\label{S4}
In this section we work on an adaptation of the seminal result of Maz'ya and Shaposhnikova, \cite{MS}. Namely we prove the following theorem.
\begin{thm}\label{ThmMS}
Let $u\in L^p(\R^n)$ and assume that there exists $s_0\in (0,1)$ such that $u\in W^{s_0,p}_i(\R^n)$.  Then
$$
\lim_{s\to 0}J^i_{s,p}(u) = 2p^{-1}\|u\|_{p}^{p}.
$$
\end{thm}

\begin{proof}
In \cite{MS}, It is proved that, for any $n\in\N$, 
\begin{equation}\label{Maz'ya}
    \lim_{s\to 0}s\intRn\intRn\frac{|u(x)-u(y)|^p}{|x-y|^{n+sp}}dxdy=2p^{-1}|S^{n-1}|\|u\|_{L^p(\R^n)}^p.
\end{equation}
for all $u \in W^{s_0,p}(\R^n)$ for some $s_0>0$.

In particular, if we apply this result in dimension $1$ to $v(x_1)=u(x_1,x')$ where $x'\in \R^{n-1}$ is fixed, we obtain that
$$
\lim_{s\to 0} s\int_\R\int_\R \frac{|v(x_1+h) - v(x_1)|^p}{|h|^{1+sp}}\, dx_1 dh = \frac{4}{p} \int_\R |v(x_1)|^p\, dx_1,
$$
which in terms of $u$ reads
\begin{equation}\label{MSdim1}
\lim_{s\to 0} s\int_\R\int_\R \frac{|u(x_1+h,x') - u(x_1,x')|^p}{|h|^{1+sp}}\, dx_1 dh = \frac{4}{p} \int_\R |u(x_1,x')|^p\, dx_1,
\end{equation}

So, to conclude our result we need to integrate with respect to $x'\in\R^{n-1}$ and for that it remains to find an integrable majorant for
$$
V_s(x') := s\int_\R\int_\R \frac{|u(x_1+h,x') - u(x_1,x')|^p}{|h|^{1+sp}}\, dx_1 dh.
$$
For that purpose, we proceed as follows
$$
V_s(x')=s\left(\int_\R\int_{|h|\leq 1}+\int_\R\int_{|h|\ge 1}\right)\frac{|u(x_1+h,x') - u(x_1,x')|^p}{|h|^{1+sp}}\, dx_1 dh = I + II.
$$
For the first term
\begin{align*}
I&=s\int_\R\int_{|h|\leq 1}\frac{|u(x_1+h,x') - u(x_1,x')|^p}{|h|^{1+sp}}\, dx_1 dh\\
&\leq s_0\int_\R\int_{|h|\leq 1}\frac{|u(x_1+h,x') - u(x_1,x')|^p}{|h|^{1+s_0p}}\, dx_1 dh
\end{align*}
The inequality came from $s\leq s_0$, and thanks to $u\in W^{s_0,p}(\R^n)$ we conclude the integrability of this part.
For the second term
\begin{align*}
II &= s\int_\R\int_{|h|\geq 1}\frac{|u(x_1+h,x') - u(x_1,x')|^p}{|h|^{1+sp}}\, dx_1 dh\\
&\leq s\int_\R 2^p|u(x_1,x')|^p \int_{|h|\geq 1}\frac{1}{|h|^{1+sp}}\,dh dx_1\\
&\leq \frac{2^p}{p}\int_\R|u(x_1,x')|^p\, dx_1
\end{align*}
and this last term is integrable in $\R^{n-1}$ since $u\in L^p(\R^n)$.

These two bounds conclude the proof.
\end{proof}

\section{Applications to partial differential equations.}\label{S5}

In this section, we will make use of the results in Section \ref{S3} to find connections between nonlocal pseudo $p-$Laplace operators and their local counterparts.

First, we will give some definitions and enforce notations that will be helpful.

\begin{defn}
Let $\Omega\subset \R^n$ be a bounded subset. The topological dual spaces of $W^{\s,\p}_0(\Omega),W_{i,0}^{s,p}(\Omega)$  will be denoted by $W^{-\s,\p'}(\Omega) $ and $W_i^{-s,p'}(\Omega)$ respectively. 
\end{defn}

Now, it is easy to see that the functionals $J^i_{s,p}$ are Fréchet differentiable for every $i=1,\dots,n$, every $s\in (0,1]$ and every $p\in (1,\infty)$. Even more  $(J^i_{s,p})'\colon W^{s,p}_{i,0}(\Omega) \to W^{-s,p'}_i(\Omega)$ is continuous and  is given by
\begin{align*}
&\langle (J^i_{s,p})'(u),v\rangle=\\
&s(1-s)p\intRn\intR\frac{|u(x+he_i)-u(x)|^{p-2}(u(x+he_i)-u(x))(v(x+he_i)-v(x))}{|h|^{1+sp}}\,dh dx,
\end{align*}
if $0<s<1$ and
$$
\langle (J^i_{1,p})'(u), v\rangle= 2\intRn|\partial_{x_i}u|^{p-2}\partial_{x_i}u \ \partial_{x_i}v \,dx,
$$
where $\langle\cdot, \cdot\rangle$ denotes the duality pairing between $W^{s,p}_{i,0}(\Omega)$ and its dual $W^{-s,p'}_i(\Omega)$.

\subsection{Pseudo fractional $\p-$Laplace operator}
Let $s_i\in (0,1], i=1,\dots,n$ be fractional parameters and define $\s=(s_1,\dots,s_n)$. Next, let $p_i\in (1,\infty), i=1,\dots,n$ be integrable parameters and define  $\p=(p_1,\dots,p_n)$.

Recall that the operator $J_{\s,\p}\colon W^{\s,\p}_0(\Omega)\to\R$ that was defined as $J_{\s,\p} := \sum_{i=1}^n\frac{1}{p_i} J^i_{s_i,p_i}$, is then Fr\'echet differentiable with derivative given by 
$$
J_{\s,\p}'\colon W^{\s,\p}_0(\Omega) \to W^{-\s,\p'}(\Omega)%:= \sum_{i=1}^n W_i^{-s_i, p_i'}(\Omega),
$$
$$
J_{\s,\p}'(u) = \sum_{i=1}^n \frac{1}{p_i}(J^i_{s_i,p_i})'(u).
$$

 So, we define the \textit{pseudo anisotropic $\p-$Laplace operator} as
 $$
 (-\widetilde{\Delta}_{\p})^{\s} u  := (J_{\s,\p})'(u).
 $$

Observe that in the case where $\s=(1,\dots,1)$ (that is the pure local case) and $\p=(p,\dots,p)$, then
$$
 (-\widetilde{\Delta}_{\p})^{\s} u  := \sum_{i=1}^n \partial_{x_i}\left(|\partial_{x_i}u|^{p-2}\partial_{x_i} u\right),
 $$
that is the well-known pseudo p-Laplace operator.

Therefore, we want to study the equation
\begin{equation}\label{EDP-anisotropica}
\begin{cases}
(-\widetilde{\Delta}_{\p})^{\s} u = f & \text{in }\Omega\\
u=0 & \text{in } \R^n\setminus \Omega
\end{cases}
\end{equation}
and, we say that $u \in \Wsp(\Omega)$ is a weak solution of \eqref{EDP-anisotropica} if 
$$
\langle(-\widetilde{\Delta}_{\p})^{\s}u,v\rangle=\int_\Omega f v\,dx
$$
for all $v \in W^{\s,\p}_0(\Omega)$.

In order for a weak solution to be well defined, as usual we need to impose some integrability conditions on the source term $f$.

We consider first the case where $f=f(x)$ and then, the Sobolev immersion theorem proved in \cite{CKM} requires that $f\in L^{(\p^*)'}(\Omega)$.

Existence and uniqueness of weak solutions to \eqref{EDP-anisotropica} is then a direct consequence of the direct method in the calculus of variations, since the solution is the unique minimizer of the functional
$$
I(v) := J_{\s,\p}(v) - \int_\Omega fv\, dx,
$$
which is a strictly convex, coercive and continuous functional in $W^{\s,\p}_0(\Omega)$.

Let us summarize all of this facts into a single statement.
\begin{prop}\label{existence1}
Let $\Omega\subset \R^n$ be a bounded open set, let $\s=(s_1,\dots,s_n)$, $s_i\in (0,1)$, $\p=(p_1,\dots,p_n)$, $p_i\in (1,\infty)$ and  $\p^*$ the critical Sobolev exponent defined in \eqref{p*}. Assume that $p_\text{max}<\p^*$.

Then, for every $f\in L^{(\p^*)'}(\Omega)$, there exists a unique weak solution $u\in W^{\s,\p}_0(\Omega)$ to \eqref{EDP-anisotropica}.
\end{prop}

\subsection{Asymptotic behavior of solutions}
The objective of this section is to analyze the asymptotic behavior of the solution to \eqref{EDP-anisotropica} when (some of the) $s_i\to 1$.

To this end assume that $s_1^k,\dots, s_j^k\to 1$ (as $k\to\infty$) and $s_{j+1},\dots,s_n$ remain fixed and consider $\s_k := (s_1^k,\dots,s_j^k,s_{j+1},\dots,s_n)$ and the functionals
$$
I_k(v) := J_{\s_k,\p}(v) - \int_\Omega fv\, dx.
$$
By Proposition \ref{existence1}, there exists a unique minimizer $u_k$ of $I_k$ in $W^{\s_k,\p}_0(\Omega)$. Hence we want to study the behavior of the sequence $\{u_k\}_{k\in\N}$ as $k\to\infty$.

To this end, let us define $\s_0 := (\underbrace{1,\dots,1}_{j},s_{j+1},\dots,s_n)$ and
$$
I_0(v) := J_{\s_0,\p}(v) - \int_\Omega fv\, dx,
$$
where 
$$
J_{\s_0,\p}(v) := \sum_{i=1}^j \frac{1}{p_i}J^i_{1,p_i}(v) + \sum_{i=j+1}^n \frac{1}{p_i}J^i_{s_i,p_i}(v).
$$

Next we make full use of the results in Section 3 and conclude the following theorem:
\begin{thm}
With the preceding notation, the functionals $\bar I_k\colon L^{p_\text{min}}(\Omega)\to \bar\R$ defined as
$$
\bar I_k(v) =\begin{cases}
I_k(v) & \text{if } v\in W^{\s_k,\p}_0(\Omega)\\
\infty & \text{elsewhere}
\end{cases}
$$
Gamma-converges to $\bar I\colon L^{p_\text{min}}(\Omega)\to \bar\R$ defined as
$$
\bar I(v) =\begin{cases}
I(v) & \text{if } v\in W^{\s_0,\p}_0(\Omega)\\
\infty & \text{elsewhere}
\end{cases}
$$
\end{thm}

\begin{proof}
By Corollary \ref{Gamma2} we have that $J_{\s_k,\p}$ Gamma converge to $J_{\s_0,\p}$ as $k\to\infty$. So the result follows since the functional $u\mapsto \int_\Omega fu\, dx$ is continuous in $L^{p_\text{min}}(\Omega)$ (see \cite{DalMaso}).
\end{proof}

An immediate consequence of this result is the convergence of the solutions of \eqref{EDP-anisotropica} to the corresponding limiting problem.
\begin{cor}\label{fixed f}
Let $u_k\in W^{\s_k,\p}_0(\Omega)$ be the weak solution to 
$$
\begin{cases}
(-\widetilde{\Delta}_\p)^{\s_k} u_k = f & \text{in }\Omega\\
u_k=0 & \text{in } \R^n\setminus\Omega,
\end{cases}
$$
then $u_k\to u_0$ in $L^{p_\text{min}}(\Omega)$ where $u_0\in W^{\s_0,\p}_0(\Omega)$ is the weak solution to 
$$
\begin{cases}
(-\widetilde{\Delta}_\p)^{\s_0} u_0 = f & \text{in }\Omega\\
u_0=0 & \text{in } \R^n\setminus\Omega,
\end{cases}
$$
\end{cor}

\subsection{The semilinear-like case}
In this subsection, we investigate the case where the source term $f$ depends also on the solution itself $u$. That is $f=f(x,u)$.

In this case, in order to have the notion of weak solution some conditions on $f$ has to be imposed. Namely:
\begin{enumerate}[label=(\subscript{f}{{\arabic*}})]
    \item   $f:\Omega \to \R\times \R$ is Carathéodory function.\label{f1}
    \item  There exist a constant $C>0$  and an exponent $q\le \p^*$ such that \label{f2}
    $$
    |f(x, z)|\le C (1+|z|^{q-1})
    $$
 \end{enumerate}

Recall that \ref{f1} implies that if $f(x,u(x))$ is measurable whenever $u(x)$ is measurable. Moreover,  \ref{f2} implies that if $u\in L^{\p^*}(\Omega)$, then $f(x,u(x))\in L^{(\p^*)'}(\Omega)$. Hence, we say that $u$ is a weak solution to
\begin{equation}\label{eq.semilineal}
\begin{cases}
(-\widetilde{\Delta}_\p)^{\s} u = f(x,u) & \text{in }\Omega\\
u=0 & \text{in } \R^n\setminus\Omega,
\end{cases}
\end{equation}
if $u\in W^{\s,\p}_0(\Omega)$ and
$$
\langle J_{\s,\p}'(u), v\rangle = \int_\Omega f(x,u(x)) v(x)\, dx,\quad \text{for every } v\in W^{\s,\p}_0(\Omega).
$$

Under these conditions, we can prove a stability result for weak solutions to \eqref{eq.semilineal}.
\begin{thm}\label{stability.semilinal}
Assume that for each $k\in \N$, there exists $u_k\in W^{\s_k,\p}_0(\Omega)$ a weak solution to 
\begin{equation}\label{eq.semilineal.k}
\begin{cases}
(-\widetilde{\Delta}_\p)^{\s_k} u_k = f(x,u_k) & \text{in }\Omega\\
u_k=0 & \text{in } \R^n\setminus\Omega,
\end{cases}
\end{equation}
Moreover, assume that the sequence $\{u_k\}_{k\in\N}$ is precompact in $L^{p_\text{min}}(\Omega)$, bounded in $L^q(\Omega)$ and that $\s_k\to \s_0$ as $k\to\infty$ and finally assume that the exponent $q$ in \ref{f2} satisfy that $q<\p^*_{\s_0}$. Hence, any accumulation point $u$ of the sequence $\{u_k\}_{k\in\N}$ belongs to $W^{\s_0,\p}_0(\Omega)$ and it is a weak solution to
$$
\begin{cases}
(-\widetilde{\Delta}_\p)^{\s_0} u = f(x,u) & \text{in }\Omega\\
u=0 & \text{in } \R^n\setminus\Omega,
\end{cases}
$$
\end{thm}

In order to prove Theorem \ref{stability.semilinal}, a couple of lemmas, that are straightforward adaptation of \cite[Lemmas 2.7 and 2.8]{FBS}.

\begin{lema}\label{estabilidad1}
With the same assumptions and notations as in Theorem \ref{stability.semilinal}, let $u \in W_0^{\s_0,\p}(\Omega)$ be fixed and for any $k\in \N$,  let $v_k\in W_0^{\s_k,\p}(\Omega)$ be such that $\sup_{k\in\N} J_{\s_k,\p}(v_k)<\infty$. Then, for $t>0$,
$$
J_{\s_k,\p}(u+tv_k)\leq J_{\s_k,\p}(u)+t\langle(-\widetilde{\Delta}_{\p})^{\s_k}u,v_k\rangle+ o(t),
$$
where $o(t)$ is uniform in $k\in\N$.
\end{lema}

\begin{lema}\label{2 estabilidad}
Under the same assumptions that Lemma \ref{estabilidad1}, if moreover $v_k\to v$ in $L^{p_\text{min}}(\Omega)$. Then, for every $u \in W_0^{\s_0,\p}(\Omega)$ we have
$$
\langle (J_{\s_k,\p})'(u),v_k\rangle \to \langle (J_{\s_0,\p})'(u), v\rangle.
$$
\end{lema}

With the help of these two lemmas, we can now proceed with the proof of Theorem \ref{stability.semilinal}.
\begin{proof}[Proof of Theorem \ref{stability.semilinal}]
First observe that
$$
\langle (J_{\s_k,\p})'(u_k),u_k\rangle \ge p_\text{min} J_{\s_k,\p}(u_k).
$$
Hence, using \ref{f2}, and the fact that $u_k$ are weak solutions to \eqref{eq.semilineal.k}, we get that
$$
J_{\s_k,\p}(u_k)\le C(1+\|u_k\|_q^q).
$$
Therefore, since $\{u_k\}_{k\in\N}$ is bounded in $L^q(\Omega)$, we have that 
\begin{equation}\label{uniformJ}
\sup_{k\in\N} J_{\s_k,\p}(u_k)<\infty.
\end{equation}

Also, since $u_k\to u$ in $L^{p_\text{min}}(\Omega)$, passing to a subsequence, if necessary, we get that $u_k\to u$ a.e. in $\Omega$ and, moreover, using Theorem \ref{compactness}, we obtain that $u \in W_0^{\s_0,\p}(\Omega)$.

Next, we define $\eta_k=f(x,u_k(x))=(-\widetilde{\Delta}_{\p})^{\s_k} u_k \in W^{-\s_k,\p'}(\Omega)$. Let us see that $\{\eta_k\} $ is bounded in  $W^{-\s_0,\p'}(\Omega)$.

First, notice that  $\p_{\s_k}^*\to \p_{\s_0}^*$ as $k \to \infty$ and $q<\p_{\s_0}^*$, then $(q-1)(\p_{\s_0}^*)'< \p_{\s_k}^*$, for $k$ large. Then
\begin{align*}
    \langle \eta_k, v\rangle&=\int_\Omega f(x,u_k)v\,dx  
    \\ & \leq C\int_\Omega(1+|u_k|^{q-1}) |v|\,dx\leq C\int_\Omega |v|\,dx+C\int_\Omega |u_k|^{q-1}|v|\,dx
    \\ &\leq C\|v\|_{\p_{\s_0}^*} + C\|v\|_{\p_{\s_0}^*} \left(\int_\Omega |u_k|^{(q-1)(\p_{\s_0}^*)'}\, dx\right)^{1/(\p_{\s_0}^*)'}.
\end{align*}
Since $(q-1)(\p_{\s_0}^*)'< \p_{\s_k}^*$, by the Sobolev immersion for the fractional anisotropic spaces proved in \cite{CKM} and the uniform bound on $J_{\s_k,\p}(u_k)$, \eqref{uniformJ}, we obtain that
$$
\|u_k\|_{(q-1)(\p_{\s_0}^*)'}\le C,\quad \text{for every } k\in\N.
$$
Hence, $\eta_k$ is bounded in $W^{-\s_0,\p'}(\Omega)$.

Therefore, up to some subsequence, there exists $\eta\in W^{-\s_0,\p'}(\Omega)$ such that $\eta_k \rightharpoonup \eta$ weakly in $W^{-\s_0,\p'}(\Omega)$.

Since $u_k$ is a weak solution of \eqref{eq.semilineal.k}, for any $v \in W^{\s_0,\p}_0(\Omega)$,
$$
0=\langle \eta_k,v\rangle-\int_\Omega f(x,u_k)v dx
$$
taking limit as $k \to \infty$ the first term converge weakly, while the second term converge  by the dominated convergence theorem and \cite[Theorem 4.9]{Brezis},
$$
0=\langle \eta,v\rangle-\int_\Omega f(x,u)v dx.
$$
Now, we want to identify $\eta$, let us see that 
$$
    \langle\eta,v\rangle=\langle(-\widetilde{\Delta}_\p)^{\s_0} u,v\rangle.
$$
For that purpose, we will use the monotonicity of the operator and the fact that $u_k$ is a weak solution of \eqref{eq.semilineal.k}. In fact,
\begin{equation}
    \begin{split}
        0&\leq\langle(-\widetilde{\Delta}_{\p})^{\s_k}u_k,u_k-v\rangle-\langle(-\widetilde{\Delta}_{\p})^{\s_k}v,u_k-v\rangle\\
        &=\int_\Omega f(x,u_k)(u_k-v)dx-\langle(-\widetilde{\Delta}_{\p})^{\s_k}v,u_k-v\rangle
    \end{split}\nonumber
\end{equation}
Taking limit $k\to \infty$ and for Lemma \ref{2 estabilidad},
\begin{equation}
     0\leq \int_\Omega f(x,u)(u-v)dx-\langle(-\widetilde{\Delta}_{\p})^{\s_0}v,u-v\rangle
  \nonumber
\end{equation}
If we take $v=u-tw,\ w \in W^{\s_0,\p}(\Omega)$ given and $t>0$ , we obtain that 
$$
0\leq \langle \eta,w\rangle-\langle-\widetilde{\Delta}_{\p})^{\s_0}(u-tw),w\rangle.
$$
Finally, taking $t\to 0$
$$
0\leq \langle \eta,w\rangle-\langle-\widetilde{\Delta}_{\p})^{\s_0}(u),w\rangle.
$$
The proof is now completed.
\end{proof}

\subsection{Ground-state solutions}
The purpose of this subsection is to investigate further the results of the preceding one and analyze the case where the sequence of solutions $\{u_k\}_{k\in\N}$ are {\em ground-state solutions} to the problem:
\begin{equation}\label{semilineal.ground.k}
    \begin{cases}
    
(-\widetilde{\Delta}_\p)^{\s_k} u_k = f(x,u_k) & \text{in }\Omega\\
u_k=0 & \text{in } \R^n\setminus\Omega,

    \end{cases}
\end{equation}
and its limit behavior to
\begin{equation}\label{semilineal.ground}
    \begin{cases}
    
(-\widetilde{\Delta}_\p)^{\s_0} u = f(x,u) & \text{in }\Omega\\
u=0 & \text{in } \R^n\setminus\Omega,

    \end{cases}
\end{equation}

Recall  that ground state solutions are minimizers of the energy functional 
$$
I_{\s,\p}(u)= J_{\s,\p}(u)- \int_\Omega F(x,u)\, dx,
$$
restricted to the {\em Nehari manifold}
$$
\mathcal{N}_{\s,\p}= \{v \in W^{\s,\p}_0(\Omega)\setminus \{0\}:\langle (I_{\s,\p})'(v),v\rangle=0\}.
 $$
where $F(x,z)=\int_0^z f(x,t)\,dt$ is the primitive of $f.$

On the nonlinearity $f$, besides \ref{f1} and \ref{f2}, it will be assumed to fulfill the following further structural hypothesis that are standard when consider ground state solutions in nonlinear problems

\begin{enumerate}[label=(\subscript{f}{{\arabic*}}),resume]
    \item $\displaystyle\frac{F(x,z)}{|z|^{p_\text{max}}}\to \infty$ as $|z|\to \infty$ uniformly with respect to $x \in \Omega$.\label{f3}
    \item $ f(x,z)=\circ(|z|^{p_{max}-1})$ as $z\to 0$ uniformly with respect to $x \in \Omega$.\label{f4}
    \item For almost every $x\in \Omega$
    $$
    \frac{f(x,z)}{|z|^{p_\text{max}-1}} \text{ is strictly increasing on }(-\infty,0)\cup(0,\infty)
    $$\label{f5}
    \item 
    There is a $\mu>p_\text{max}$ such that
    $$
    \mu F(x,z)\leq z f(x,z)
    $$ \label{f6}
\end{enumerate}
Also, as in the preceding subsection, the exponent $q$ in \ref{f2} must be subcritical, i.e. $q<\p^*$.

 It is well known that under hypotheses \ref{f1}-\ref{f6} a ground state actually exists. In fact, in our context, a ground state solution, $u_\s$, is a mountain pass solution.
 
\begin{rem}\label{homeoNS}
Observe that the Nehari manifold $\mathcal{N}_{\s,\p}$ has the property that given any nonzero function $u\in W^{\s,\p}_0(\Omega)$, there exists a unique $t>0$ such that $tu\in \mathcal{N}_{\s,\p}$. We denote this scalar as $t_u^\s$.

The existence and uniqueness of this scalar $t^\s_u$, follows by using standard arguments in the Calculus of Variations. That is,  given $t\in \R$ and $u\in W^{\s,\p}_0(\Omega)$ define the function
$$
\phi(t)=J_{\s,\p}(tu)-\int_\Omega F(x,tu)\,dx.
$$
On the one hand, it is easy to verify that $\phi(0)=0$, $\phi'(t)\geq 0$ for small values of $t>0$ (these follows from the definition of $\phi$ and \ref{f4}). Moreover, by \ref{f3} we obtain that $\phi(t)\to-\infty$ when $t\to \infty.$ So we conclude the existence of at least one critical point of $\phi.$

On the other hand, $\phi'(t)$ verifies that
\begin{align*}
    \phi'(t)&=\sum_{i=1}^n t^{p_i} J^i_{s_i,p_i}(u)-\int_\Omega f(x,ut)tu\,dx\\
    &=t^{p_\text{max}}\Big[C+\underbrace{\sum_{p_i\not={p_\text{max}}}t^{p_i-p_\text{max}} J^i_{s_i,p_i}(u)-\int_\Omega \frac{f(x,ut)u^{p_\text{max}}}{|ut|^{p_\text{max}-1}}\,dx}_{A}\Big].
\end{align*}
By \ref{f5}, $A$ is strictly decreasing, so the critical point must be unique.

\end{rem}

By the previous remark, it is easy to check that a ground state solution must satisfy 
$$
c_\s=I_{\s,\p}(u_\s)= \inf_{v\in W^{\s,\p}_0(\Omega)\setminus\{0\}}\ \sup_{t>0} I_{\s,\p}(tv)>0
$$

The next lemma shows that the mountain pass levels of the ground state solutions of \eqref{semilineal.ground.k} are uniformly bounded.
\begin{lema}\label{levels.mountain.pass}
Under the previous notation, $c_{\s_0}\geq \limsup_{\s\to \s_0}c_\s$.
\end{lema}

\begin{proof}
Let $u\in \mathcal{N}_{\s_0,\p}$, then $t^\s_u u\in \mathcal{N}_{\s,\p}$. We first need to analyze the behavior of $t_u^\s$ as $\s\to\s_0$.

Since $t_u^\s u\in \mathcal{N}_{\s,\p}$, we have that
\begin{align*}
    \langle (J_{\s,\p})'(t_u^\s u), t_u^\s u\rangle &= \sum_{i=1}^n |t_u^\s|^{p_i} J^i_{s_i,p_i}(u) \\
    &= \int_\Omega f(x,t_u^\s u)t_u^\s u\,dx\\ 
    &\geq \mu \int_\Omega F(x,t_u^\s u)\,dx\\
    &= \mu |t_u^\s|^{p_\text{max}}\int_\Omega \frac{F(x,t_\s u) |u|^{p_\text{max}}}{|t_u^\s u|^{p_\text{max}}}\, dx.
\end{align*}
Now, if we assume that $t_u^\s\ge 1$, we obtain that
$$
p_\text{max}J_{\s,\p}(u)\ge \mu \int_\Omega \frac{F(x,t_\s u) |u|^{p_\text{max}}}{|t_u^\s u|^{p_\text{max}}}\, dx 
$$

As, by Theorem \ref{BBM1},  $J_{\s,\p}(u)\to J_{\s_0,\p}(u)<\infty$  when $\s\to \s_0$, from \ref{f3} implies that $\{t_u^\s\}_\s$ is bounded.

Let $t_0\geq 0$ be any accumulation point of $\{t_u^\s\}_\s$. Let $t_k = t_u^{\s_k}$ be a subsequence such that $t_k \to t_0$ as $k\to \infty$. Let us see that $t_0\neq 0$. In fact, the Nehari identity gives us
$$
\sum_{i=1}^n |t_k|^{p_i} J^i_{s_i,p_i}(u) = \int_\Omega f(x,t_k u)t_k u\,dx
$$
from where we get
$$
\sum_{i=1}^n |t_k|^{p_i-p_\text{min}} J^i_{s_i,p_i}(u) = \int_\Omega \frac{f(x,t_k u)}{|t_k u|^{p_\text{min}-1}}|u|^{p_\text{min}-1}u\,dx
$$
Now, if $t_0=0$ and $p_\text{min} = p_{i_0}$, passing to the limit as $k\to\infty$ using \ref{f4} we deduce that 
$$
J_{s_{i_0},p_\text{min}}^{i_0}(u) = 0.
$$
But this implies that $u$ is independent of $x_{i_0}$ and thus that $u=0$, a contradiction. Therefore $t_0>0$.

Moreover, from \ref{f2} and \ref{f4} we have that 
$$
|f(x,t_k u)t_k u|\leq \epsilon |t_k u|^{p_\text{min}} + C_\epsilon |t_k u|^q\leq C(|u|^{p_\text{min}}+|u|^q) \in L^1(\Omega),
$$
where $C>0$ is independent of $k$. Then, using the dominated convergence theorem ,
$$
\int_\Omega f(x,t_k u) t_k u\,dx \to \int_\Omega f(x,t_0 u)t_0u\,dx \quad (\text{as } k\to\infty).
$$
So, we get 
$$
\sum_{i=1}^n |t_0|^{p_i} J^i_{s_i^0,p_i}(u) = \int_\Omega f(x, t_0 u) t_0 u\,dx.
$$
Since $u\in \mathcal{N}_{\s_0,\p},$ $t_0=1$. Then  $t_u^\s\to 1$ as $\s\to \s_0$.

Now, we proceed as follows:
\begin{align*}
c_\s\le I_{\s,\p}(t_u^\s u) &= \sum_{i=1}^n \frac{1}{p_i} J^i_{s_i,p_i}(t_u^\s u) - \int_\Omega F(x,t_u^\s u)\, dx\\
&= \sum_{i=1}^n \frac{1}{p_i} |t_u^\s|^{p_i} J^i_{s_i,p_i}(u) - \int_\Omega F(x,t_u^\s u)\, dx.
\end{align*}
We can now take the limit as $\s\to\s_0$ in the former inequality to obtain
$$
\limsup_{\s\to\s_0} c_\s \le \sum_{i=1}^n \frac{1}{p_i} J^i_{s_i^0,p_i}(u) - \int_\Omega F(x,u)\, dx = I_{\s_0,\p}(u),
$$
where we have used Theorem \ref{BBM1}.

Then, we take the infimum in $u\in \mathcal{N}_{\s_0,\p}$ to conclude the desired result.
\end{proof}

\begin{lema}\label{lemma 5.11}
Let $u_s \in W^{\s,\p}_0(\Omega)$ be a ground state solution of \eqref{semilineal.ground.k}.
Then there exist two constants $0<c<C<\infty$ independent on $s$ such that 
$$
c \leq J_{\s,\p}(u_\s)\leq C
$$
\end{lema}
\begin{proof}
Let $u_\s\in W^{\s,\p}_0(\Omega)$ be a ground state solution of \eqref{semilineal.ground.k}. By the previuos lemma, there exist a constant $C>0$, such that
\begin{align*}
    C&\geq I_{\s,\p}(u_\s)-\frac{1}{\mu}\langle I_{\s,\p}'(u_\s),u_\s\rangle
    \\&= J_{\s,\p}(u_\s)-\sum_{i=1}^n \frac{1}{\mu} J^i_{s_i, p_i}(u_\s) - \int_\Omega \left(F(x,u_\s)-\frac{1}{\mu}f(x,u_\s)u_\s\right)\, dx\\
    &\geq \sum_{i=1}^n\left( \frac{1}{p_i} - \frac{1}{\mu}\right)J^i_{s_i,p_i}(u_s)\\
    &\geq \left(\frac{1}{p_\text{max}}-\frac{1}{\mu}\right) J_{\s,\p}(u).
\end{align*}

For the lower bound, since $u_\s\in \mathcal{N}_{\s,\p}$, we have that 
\begin{equation}\label{estimate.prem}
p_\text{min}J_{\s,\p}(u_\s)\le \int_\Omega f(x,u_\s)u_\s\,dx\leq \int_\Omega \left(\epsilon |u_\s|^{p_\text{max}}+C_\epsilon|u_\s|^q\right)\,dx.
\end{equation}
Now from \cite[Theorem 1.2]{CKM}, we obtain that
\begin{equation}\label{inmersion}
\min\{\|u\|_{\p^*}^{p_\text{min}}, \|u\|_{\p^*}^{p_\text{max}}\}\le C J_{\s,\p}(u),
\end{equation}
where $C$ is independent of $u\in W^{\s,\p}_0(\Omega)$.

Next, from \eqref{estimate.prem}, using H\"older's inequality and \eqref{inmersion}, we obtain
\begin{align*}
J_{\s,\p}(u_\s)\le& C\epsilon \max\{J_{\s,\p}(u_\s), J_{\s,\p}(u_\s)^{p_\text{max}/p_\text{min}}\}\\
& + C_\epsilon \max\{J_{\s,\p}(u_\s)^{q/p_\text{max}}, J_{\s,\p}(u_\s)^{q/p_\text{min}}\}.
\end{align*}
Then, choosing $\epsilon>0$ small, we obtain that $J_{\s,\p}(u_\s)$ is bounded below away from zero, as we wanted to show.
\end{proof}

Now we are ready to prove the main result of the section.
\begin{thm}\label{stability.ground}
Let $\s_k\to \s_0$ as $k\to\infty$ and let $\p$ be such that $1<p_\text{min}\le p_\text{max}<\p^*_0=\p_{\s_0}^*$.
Assume that for each $k\in \N$, there exists $u_k\in W^{\s_k,\p}_0(\Omega)$ a ground state solution to \eqref{semilineal.ground.k}.

Then the sequence $\{u_k\}_{k\in\N}$ is precompact in $L^r(\Omega)$, for every $1<r<\p^*_0$ and any accumulation point $u$ of the sequence $\{u_k\}_{k\in\N}$ belongs to $W^{\s_0,\p}_0(\Omega)$ and it is a ground state solution to
\eqref{semilineal.ground}.
\end{thm}

\begin{proof}
First observe that if $\s_k\to\s_0$, then $\p^*_k := \p_{\s_k}^*\to \p_0^*$ as $k\to\infty$. Then, if $r<\p^*_0$, then $r<\p^*_k$ for every $k$ large.

In view of \eqref{inmersion} and H\"older's inequality, if we show that $\sup_{k\in\N} J_{\s_k,\p}(u_k)<\infty$, then $\{u_k\}_{k\in\N}$ will be bounded in $L^r(\Omega)$.

But, this follows from Lemma \ref{lemma 5.11} using that $f$ satisfies \ref{f6}. Finally, Theorem \ref{compactness} gives the existence of a subsequence $\{u_{k_j}\}_{j\in\N}\subset \{u_k\}_{k\in\N}$ and a function $u\in W^{\s_0,\p}_0(\Omega)$ such that $u_{k_j}\to u_0$ in $L^r(\Omega)$.

It remains to see that $u_0$ is a ground state solution to \eqref{semilineal.ground}. But, by Theorem \ref{stability.semilinal}, $u_0$ is a weak solution of \eqref{semilineal.ground}. Now, since $u_k\in \mathcal{N}_{\s_k,\p}$, we have that 
$$
0<c\leq p_\text{min} J_{\s_k,\p}(u_k)\le \sum_{i=1}^n J^i_{s_i,p_i}(u_k)=\int_\Omega f(x,u_k)u_k\, dx \to \int_\Omega f(x,u_0)u_0\, dx,
$$
as $k\to \infty$, so $u_0\neq 0$ and it follows that $u_0\in \mathcal{N}_{\s_0,\p}$. Moreover,
\begin{align*}
\liminf_{k\to \infty} c_{\s_k}&=\liminf_{k\to \infty} I_{\s_k,\p}(u_k)=\liminf_{k\to \infty}J_{\s_k,\p}(u_k) - \int_\Omega F(x,u_k)\\
&\geq J_{\s_0,\p}(u_0) - \int_\Omega F(x,u_0)\,dx = I_{\s_0,\p}(u_0)\geq c_{\s_0},
\end{align*}
where we have used Corollary \ref{Gamma2} in the inequality.

Combining this estimate with Lemma \ref{levels.mountain.pass} we conclude that
$$
c_{\s_0}=I_{\s_0,\p}(u_0)
$$
and the proof is finished.
\end{proof}

\section*{Acknowledgements}
This work was partially supported by UBACYT Prog. 2018 20020170100445BA, by ANPCyT PICT 2019-00985 and by PIP No. 11220150100032CO.

J. Fern\'andez Bonder is a members of CONICET and I. Ceresa Dussel is a doctoral fellow of CONICET.

\bibliographystyle{plain}
\bibliography{References.bib}
\end{document}